\documentclass[10pt]{article}
\usepackage{amsmath,amsthm}
\usepackage{amsfonts}
\usepackage{amssymb}
\usepackage{amscd}
\usepackage{diagrams}
\usepackage{graphicx}
\usepackage{color}
\usepackage{bbm}
\usepackage{anysize}
\usepackage{url}
\usepackage{paralist}
\usepackage[all,cmtip]{xy}
\usepackage[utf8]{inputenc}
\usepackage{hyperref}
   
\theoremstyle{plain}
\newtheorem{theorem}{Theorem}[section]

\newtheorem{lemma}[theorem]{Lemma}
\newtheorem{corollary}[theorem]{Corollary}

\newtheorem{conjecture}[theorem]{Conjecture}

\theoremstyle{definition}

\newtheorem{observation}[theorem]{Observation}
\newtheorem{remark}[theorem]{Remark}    
   
\newcommand{\barany}{B\'{a}r\'{a}ny}
\newcommand{\blagojevic}{Blagojevi\'{c}}

\newcommand{\cech}{\v{C}ech}
\newcommand{\matousek}{Matou\v{s}ek}

\newcommand{\vrecica}{Vre\'{c}ica}

\newcommand{\zivaljevic}{\v{Z}ivaljevi\'{c}}

\newcommand{\R}{\mathbbm{R}} 
\newcommand{\Z}{\mathbbm{Z}} 
\newcommand{\ZZ}{\mathbbm{Z}} 
\newcommand{\wo}{\backslash} 
\newcommand{\To}{\longrightarrow} 
\newcommand{\toG}[1]{\longrightarrow_{#1}} 
\newcommand{\tof}[1]{\stackrel{#1}{\longrightarrow}} 
\newcommand{\im}{\textnormal{im}} 

\newcommand{\iso}{\cong} 
\newcommand{\conv}{\textnormal{conv}} 
\newcommand{\one}{\mathbbm{1}} 
\newcommand{\ind}{\textnormal{Ind}}
\newcommand{\pt}{\textnormal{pt}} 
\newcommand{\incl}{\hookrightarrow} 
\newcommand{\CC}{C}
\newcommand{\qm}[1]{``{#1}''} 

\newcommand{\eps}{\varepsilon}
\newcommand{\rank}{\textnormal{rank}}

\newcommand{\id}{\textnormal{id}}
\newcommand{\vertices}{\textnormal{vert}}

\newcommand{\simplex}{\sigma}

\newcommand{\conn}{\textnormal{conn}}

\newenvironment{myitemize}
{\begin{list}{\labelitemi}{\leftmargin=1em}}
{\end{list}}
 
\newcommand\Sym{\mathfrak S}
\begin{document}

\title{\textsf{Optimal bounds for a colorful Tverberg--\vrecica\ type problem}} 
\newcommand{\smallToggle}{} 
\author{
Pavle V. M. Blagojevi\'{c}\thanks{%
Supported by the grant 144018 of the Serbian Ministry of Science and
Environment} \\
\smallToggle Mathemati\v cki Institut SANU\\
\smallToggle Knez Michailova 36\\
\smallToggle 11001 Beograd, Serbia\\
\smallToggle \url{pavleb@mi.sanu.ac.rs} \and \setcounter{footnote}{0} 
Benjamin Matschke$^{*}$%
\thanks{$^{*}$Supported by Deutsche Telekom Stiftung}\\
\smallToggle Inst.\ Mathematics, MA 6-2\\
\smallToggle TU Berlin\\
\smallToggle 10623 Berlin, Germany\\
\smallToggle \url{matschke@math.tu-berlin.de} \and \setcounter{footnote}{0}
G\"unter M. Ziegler$^{**}$%
\thanks{$^{**}$%
Partially supported by DFG,
Research Training Group ``Methods for Discrete Structures''} \\
\smallToggle Inst.\ Mathematics, MA 6-2\\
\smallToggle TU Berlin\\
\smallToggle 10623 Berlin, Germany\\
\smallToggle \url{ziegler@math.tu-berlin.de}}
\date{{\small October 26, 2009; revised December 30, 2009}}

\maketitle
\begin{abstract}\noindent
We prove the following optimal colorful Tverberg--Vre\'{c}ica type transversal theorem: 
For prime $r$ and for any $k+1$ colored collections of points $\CC^\ell$ in $\R^d$, 
$\CC^\ell=\biguplus C_i^\ell$, $|\CC^\ell|=(r-1)(d-k+1)+1$, 
$|C_i^\ell|\le r-1$, $\ell=0,\dots,k$,  there are partition of the 
collections $\CC^\ell$ into colorful sets $F_1^\ell,\dots,F_r^\ell$ 
such that there is a $k$-plane that meets all the convex hulls $\conv(F_j^\ell)$, 
under the assumption that $r(d-k)$ is even or $k=0$.

Along the proof we obtain three results of independent interest:
We present two alternative proofs for the special case $k=0$
(our optimal colored Tverberg theorem (2009)), calculate the 
cohomological index for joins of chessboard complexes, and establish a new 
Borsuk--Ulam type theorem for $(\Z_p)^m$-equivariant bundles that generalizes
results of Volovikov (1996) and \zivaljevic\ (1999).
\end{abstract}

\section{Introduction}

In their 1993 paper \cite{TV93} H. Tverberg and S. \vrecica\ presented
a conjectured common generalization of some Tverberg type theorems, 
some ham sandwich type theorems and many intermediate results.
See \cite{Ziv99} for a further collection of implications.
 
\begin{conjecture}[Tverberg--\vrecica\ Conjecture]
\label{conjTV} Let $0\le k\le d$ and let $\CC^0,\dots,\CC^k$ 
be finite point sets in~$\R^d$ of cardinality $|\CC^\ell|=(r_\ell-1)(d-k+1)+1$.
Then one can partition each $\CC^\ell$ into $r_\ell$ sets $F_1^\ell,\dots,F_{r_\ell}^\ell$ 
such that there is a $k$-plane $P$ in $\R^d$ that
intersects all the convex hulls $\conv(F_j^\ell)$, $0\le \ell\le k$, $1\le j\le r_\ell$.
\end{conjecture}

\noindent 
The Tverberg--\vrecica\ Conjecture has been verified for
the following special cases:

\medskip

\begin{compactitem}[$\ \ \bullet$]
\item $k=d$ (follows directly),
\item $k=0$ (Tverberg's theorem \cite{Tve66}),
\item $k=d-1$ (Tverberg \& \vrecica\ \cite{TV93}),
\item for $k=d-2$ a weakened version was shown in \cite{TV93} 
(one requires two more points for each~$\CC^\ell$),
\item $k$ and $d$ are odd, and $r_0=\dots=r_k$ is an odd prime (\zivaljevic\ \cite{Ziv99}),
\item $r_0=\dots=r_k=2$ (\vrecica\ \cite{Vre08}), and
\item $r_\ell=p^{a_\ell}$, $a_\ell\ge 0$, for some prime $p$, 
and $p(d-k)$ is even or $k=0$ (Karasev \cite{Kar07}).
\end{compactitem}

\medskip

\noindent
In this paper we consider the following colorful generalization of the
Tverberg--\vrecica\ conjecture.

\begin{conjecture}
\label{conjColoredTV} Let $0\le k\le d$, $r_\ell\ge 2$ $(\ell=0,\dots,k)$  
and let $\CC^\ell$\ $(\ell=0,\dots,k)$ be subsets of $\R^d$ of
cardinality $|\CC^\ell|=(r_\ell-1)(d-k+1)+1$. Let the $\CC^\ell$ be colored,
\[  
C^\ell\ =\ \biguplus C^\ell_i,
\]
such that no color class is too large, $|C^\ell_i|\le r_\ell-1$. Then we
can partition each $C^\ell$ into sets $F_1^\ell,\dots,F_{r_\ell}^\ell$
that are \emph{colorful} $($in the sense that $|C_i^\ell\cap F_j^\ell|\le1$ 
for all $i,j,\ell)$ and find a
$k$-plane $P$ that intersects all the convex hulls
$\conv(F_j^\ell)$.
\end{conjecture}

The Tverberg--\vrecica\ Conjecture \ref{conjTV} is the special case of the 
previous conjecture when all color classes are given by singletons. 
The main result of this paper is the following special case.

\begin{theorem} [Main Theorem]
\label{thmColoredTV} Let $r$ be prime and $0\le k\le d$ such that
$r(d-k)$ is even or $k=0$. Let~$\CC^\ell$\ $(\ell=0,\dots,k)$ be subsets of
$\R^d$ of cardinality $|\CC^\ell|=(r-1)(d-k+1)+1$. Let the $\CC^\ell$ be
colored,
\[
\CC^\ell\ =\ \biguplus C_i^\ell,
\]
such that no color class is too large, $|C_i^\ell|\le r-1$. Then we
can partition each $\CC^\ell$ into colorful sets $F_1^\ell,\dots,F_r^\ell$
and find a $k$-plane $P$ that intersects all the convex hulls
$\conv(F_j^\ell)$.
\end{theorem}

In Section \ref{secNewColTVThm} we will see that this theorem is quite tight 
in the sense that it becomes false if one single color class $C_i^\ell$ 
has $r_\ell$ elements and all the other ones are singletons.

Since we will prove Theorem \ref{thmColoredTV} topologically it has a 
natural topological extension, Theorem~\ref{thmTopColoredTV}.

Recently we had obtained the first case $k=0$ using 
equivariant obstruction theory \cite{BMZ09}.
In Section \ref{secNewColTopTverThmRevisited} we present two alternative proofs, 
based on the configuration space/test map scheme from \cite{BMZ09}.
The first one is more elementary and shorter; it uses a degree argument.
The second proof puts the first one into the language of cohomological index theory.
For this, we calculate the cohomological index of a join of chessboard complexes.
This allows for a more direct proof of the case $k=0$, which is the first 
of two keys for the Main Theorem~\ref{thmColoredTV}.

The second key is a new Borsuk--Ulam type theorem for equivariant bundles.
We establish it in Section~\ref{secNewBUThm}, and prove the Main Theorem 
in Section~\ref{secNewColTVThm}.
The new Borsuk--Ulam type theorem can also be applied to obtain an 
alternative proof of Karasev's above-mentioned result; see Section~\ref{secNewColTVThm}.
Karasev has also obtained a colored version of the Tverberg--\vrecica\ conjecture, 
different from ours, even for prime powers, which can also alternatively be obtained 
from our new Borsuk--Ulam type theorem. 
 
\section{The topological colored Tverberg problem revisited} \label{secNewColTopTverThmRevisited}

In \cite{BMZ09} we have shown the following new colored version of
the topological Tverberg theorem. It is the special case $k=0$ 
of the Topological Main Theorem \ref{thmTopColoredTV}.

\begin{theorem}[\cite{BMZ09}]
\label{thmNewColTopTverThm} Let $r\ge 2$ be prime, $d\ge 1$, and
$N:=(r-1)(d+1)$. Let $\simplex_N$ be an $N$-dimensional simplex with
a partition of the vertex set into \qm{color classes}
$C_0,\dots,C_m$ such that $|C_i|\le r-1$ for all $i$.

Then for every continuous map $f:\simplex_N\to \R^d$ there are $r$
disjoint faces $F_1,\dots,F_r$ of $\sigma_N$ that are colorful
$($that is, $|C_i\cap F_j|\le 1)$ such that
\[
f(F_1)\cap\dots\cap f(F_r)\neq\emptyset.
\]
\end{theorem}

\noindent
This implies the optimal colored Tverberg theorem (the \barany--Larman conjecture) for primes
minus one, even its topological extension. 
This conjecture being proven implies new complexity bounds in computational geometry; 
see the introduction of \cite{BMZ09} for three examples.

In this section we present two new proofs of Theorem \ref{thmNewColTopTverThm}. 
The first one uses an elementary degree argument. 
The second proof puts the first one into the language of cohomological index theory, 
as developped by Fadell and Husseini \cite{FH88}. 
Even though the second proof looks more difficult it actually allows 
for a more direct path, 
since it avoids the non-topological reduction of Lemma \ref{lemReductionToThm21}.
This requires more index calculations, which however are valuable since
they provide a first key step towards our 
proof of the Main Theorem \ref{thmColoredTV} in Section~\ref{secNewColTVThm}. 

\subsubsection*{The configuration space/test map scheme}

Suppose we are given a continuous map
\[
f:\simplex_N\To \R^d
\]
and a coloring of the vertex set $\vertices(\simplex_N)=[N+1]=C_0\uplus\dots\uplus C_m$ 
such that $|C_i|\le r-1$.
We want to find a colored Tverberg partition $F_1,\dots,F_r$.

As in \cite{BMZ09} we construct a test-map $F$ out of $f$.
Let $f^{*r}:(\simplex_N)^{*r}\toG{\Z_r} (\R^d)^{*r}$ be the $r$-fold join of $f$.
Since we are interested in pairwise disjoint faces $F_1,\dots,F_r$, 
we restrict the domain of $f^{*r}$ to the $r$-fold $2$-wise deleted join of $\simplex_N$,
$(\simplex_N)^{*r}_{\Delta (2)} = [r]^{*(N+1)}$.
(See \cite{Mat03} for an introduction to these notions.)
Since we are interested in colorful $F_j$s, we restrict the domain further 
to the subcomplex $[r]^{*|C_0|}_{\Delta(2)}*\dots*[r]^{*|C_m|}_{\Delta(2)}$.
The space $[n]^{*m}_{\Delta(2)}$ is known as the \emph{chessboard complex} $\Delta_{n,m}$.
We write
\begin{equation}
\label{eqDefOfTestSpaceK}
K:=(\Delta_{r,|C_0|}) * \dots * (\Delta_{r,|C_m|}).
\end{equation}
Hence we get a map
\[
F'': K \toG{\Z_r} (\R^d)^{*r}.
\]
Let $\R[\Z_r]\iso\R^r$ be the regular representation of $\Z_r$
and $W_r\subseteq \R^r$ the orthogonal complement of the all-one vector $\one=e_1+\dots+e_r$. 
The orthogonal projection $p:\R^r\toG{\Z_r} W_r$ yields a map
\[
\begin{array}{lll}
(\R^d)^{*r} &\toG{\Z_r}\ & (W_r)^{d+1}\\
\sum_{j=1}^r \lambda_j x_j &\longmapsto &
 (p(\lambda_1,\dots,\lambda_r),p(\lambda_1x_{1,1},\dots,\lambda_rx_{r,1}),\dots,
p(\lambda_1x_{1,d},\dots,\lambda_rx_{r,d}).
\end{array}
\]
The composition of this map with $F''$ gives us the test-map $F'$,
\[
F': K \toG{\Z_r} (W_r)^{d+1}.
\]
The pre-images $(F')^{-1}(0)$ of zero correspond exactly to the 
colored Tverberg partitions.
Hence the image of $F'$ contains $0$ if and only if the map $f$ 
admits a colored Tverberg partition.
Suppose that $0$ is not in the image, then we get a map
\begin{equation}
\label{eqTestmapF}
F: K \toG{\Z_r} S((W_r)^{d+1})
\end{equation}
into the representation sphere by composing $F'$ with the radial projection map.
We will derive contradictions to the existence of such an equivariant map.

The first proof establishes a key special case of 
Theorem \ref{thmNewColTopTverThm}, which implies the general result
by the following reduction. 

\begin{lemma}[\cite{BMZ09}]
\label{lemReductionToThm21}
It suffices to prove Theorem \ref{thmNewColTopTverThm} for
 $m=d+1$, $|C_0|=\dots=|C_d|=r-1$ and $|C_{d+1}|=1$.
\end{lemma}

\noindent
For the elementary proof of this lemma see 
\cite[Reduction of Thm.~2.1 to Thm.~2.2]{BMZ09}. 
This lemma is the special case $k=0$ of Lemma \ref{lemElementaryReductionLemma}, which we prove later.

Therefore we can assume that $K=K' *[r]$ and $K'=(\Delta_{r,r-1})^{*(d+1)}$.
Let $M$ be the restriction of~$F$ to~$K'$.  
The chessboard complex $\Delta_{r,r-1}$ is a connected orientable 
pseudo-manifold, hence $K'$ is one as well.
The dimensions $\dim K'=N-1=\dim S((W_r)^{d+1})$ coincide.
Thus we can talk about the degree $\deg(M)\in\Z$.
Here we are not interested in the actual sign, hence we do not need
to fix orientations.
Since $K'$ is a free $\Z_r$-space and $S^{N-1}$ is $(N-2)$-connected,
the degree $\deg(M)$ is uniquely determined modulo $r$:
This is because $M$ is unique up to $\Z_r$-homotopy on the codimension 
one skeleton of $K'$, and changing $M$ on top-dimensional cells of $K'$ 
has to be done $\Z_r$-equivariantly, hence it affects $\deg(M)$ by 
a multiple of $r$.

To determine $\deg(M)\mod r$, we let $f$ be the affine map 
that takes the vertices in $C_0$ to $-\one=-(e_1+\dots+e_d)$ 
and the vertices in $C_i\ (1\le i\le d)$ to $e_i$, where $e_i$ 
is the $i$th standard basis vector of~$\R^d$.
The singleton $C_{d+1}$ does not matter, 
we can choose it arbitrarily in $\R^d$.
Let $P\in S(W_r^{d+1})$ be the normalization of the point 
$(p(1,\dots,1,0),0,\dots,0)\in W_r^{d+1}$.
The pre-image $M^{-1}(P)$ is exactly the set of barycenters of 
the $(r-1)!^{d+1}$ top-dimensional faces of $K'\cap \Delta_{r-1,r-1}$.
With $\Delta_{r-1,r-1}$ we mean the full subcomplex 
$[r-1]^{*(r-1)}_{\Delta(2)}$ of $\Delta_{r,r-1}$.
One checks that all pre-images of $P$ have the same pre-image orientation.
This was essentially done in \cite{BMZ09}
when we calculated that $c_f(\Phi)=(r-1)!^d\zeta$.
Hence
\begin{equation}
\label{eqDegreeOfM}
\deg(M)=\pm(r-1)!^{d+1}=\pm 1\mod r.
\end{equation}

Alternatively one can take any map $m:\Delta_{r,r-1}\toG{\Z_r} S(W_r)$, 
show that its degree is $\pm 1$ by a
similar pre-image argument in dimension $d=1$, and deduce that 
\[
\deg(M)=\deg(m^{*(d+1)})=\deg(m)^{d+1}=\pm 1\mod r.
\]

\begin{proof}[First proof of Theorem \ref{thmNewColTopTverThm}]

Since $\deg(M)\neq 0$, $M$ is not null-homotopic.
Thus $M$ does not extend to a map with domain $K'*[1]\subseteq K$.
Therefore the test-map $F$ of \eqref{eqTestmapF} does not exist.
\end{proof}

\begin{remark}
The degree $\deg(M)$ is even uniquely determined modulo 
$r!$. To see this one uses the $\Sym_r$-equivariance of $M$ and the 
fact that $M$ is given uniquely up to $\Sym_r$-homotopy on the non-free part, 
which lies in the codimension one skeleton of $K'$. The latter can be shown 
with the modified test-map $F_0$ from \cite{BMZ09}. This might be an ansatz 
for the affine version of Theorem~\ref{thmNewColTopTverThm} for non-primes~$r$.
\end{remark}

\subsubsection*{Index computations}

Let in the following $H^*$ denote \cech\ cohomology with $\Z_r$-coefficients, where $r$ is prime.
The equivariant cohomology of a $G$-space $X$ is defined as
\[
H^*_G(X)\ :=\ H^*(EG\times_G X),
\]
where $EG$ is a contractible free $G$-CW complex and $EG\times_G X:=(EG\times X)/G$.
The classifying space of $G$ is $BG:= EG/G$.
If $p:X\to B$ is furthermore a projection to a trivial $G$-space $B$, 
we denote the \emph{cohomological index} of $X$ over $B$, also called the 
\emph{Fadell--Husseini index} \cite{FH88}, by
\[
\ind_G^B(X)\ :=\ \ker\big(H^*_G(B)\tof{p^*} H^*_G(X)\big)
           \ \subseteq\ H^*_G(B)\iso H^*(BG)\otimes H^*(B).
\]
If $B=\pt$ is a point then one also writes 
$H_G^*(\pt)=H^*(G)$ and $\ind_G(X):=\ind_G^\pt(X)$.
\medskip

\noindent
The cohomological index has the four properties
\begin{myitemize}
\item Monotonicity: If there is a bundle map $X\toG{G}Y$ then
\[
\ind_G^B(X)\supseteq\ind_G^B(Y).
\]
\item Additivity: If $(X_1\cup X_2,X_1,X_2)$ is excisive, then
\[
\ind_G^B(X_1)\cdot\ind_G^B(X_2)\subseteq\ind_G^B(X_1\cup X_2).
\] 
\item 
Joins: 
\[
\ind_G^B(X)\cdot\ind_G^B(Y)\subseteq\ind_G^B(X*Y).
\]
\item Subbundles: From the continuity of \cech\ cohomology $H^*$ it follows that 
if there is a is a 
bundle map $f:X\toG{G}Y$ and a closed subbundle $Z\subseteq Y$ then
\begin{equation}
\label{eqIndexThm}
\ind_G^B(f^{-1}(Z))\cdot\ind_G^B(Y)\subseteq \ind_G^B(X).
\end{equation}
\end{myitemize}

\noindent
The first two properties imply the other two. For more information about this index theory see \cite{FH87} and \cite{FH88}.

If $r$ is odd then the cohomology of~$\Z_r$ as a $\Z_r$-algebra is
\[
H^*(\Z_r)=H^*(B\Z_r)\iso \Z_r[x,y]/ {(y^2)},
\]
where $\deg(x)=2$ and $\deg(y)=1$.
If $r$ is even, then $r=2$ and $H^*(\Z_r)\iso \Z_2[t]$, $\deg t=1$.

\begin{theorem}
Let $K$ be an $n$-dimensional connected free $\Z_r$-CW complex and 
let $S$ be an $n$-dimensional $(n-1)$-connected free $\Z_r$-CW complex.
If there is a $\Z_r$-map $M:K\toG{\Z_r}S$ that induces an isomorphism on $H^n$, then
\[
\ind_{\Z_r}^\pt(K) = H^{*\ge n+1}(B\Z_r).
\]
\end{theorem}

\begin{proof}
The two fiber bundles $E\Z_r\incl E\Z_r\times_{\Z_r} K\to B\Z_r$ and 
$E\Z_r\incl E\Z_r\times_{\Z_r} S\to B\Z_r$ induce two Leray--Serre spectral sequences 
$E_*^{*,*}(K)$ and $E_*^{*,*}(S)$, and $M$ induces a morphism $M^*$ between them;
see Figure \ref{figSSofFHindexThm}.

\begin{figure}[tbh]
\centering
\input{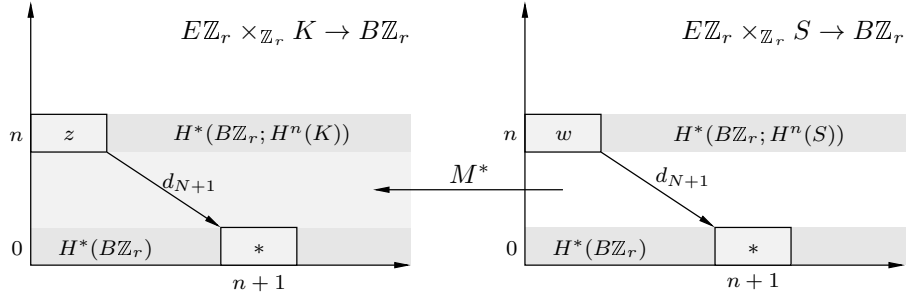}
\caption{The morphism $M^*$ between the spectral sequences $E_*^{*,*}(S)$ and $E_*^{*,*}(K)$.}
\label{figSSofFHindexThm}
\end{figure}

The two $n$th rows $H^*(\Z_r;H^n(K))$ and $H^*(\Z_r;H^n(S))$ 
at the $E_2$-pages are identified by $M^*$.
At the $E_\infty$-pages both spectral sequences have to satisfy $E_\infty^{p,q}=0$ 
whenever $p+q\ge n+1$.
This is because $K$ is free, hence $H^*_{\Z_r}(K)\iso H^*(K/{\Z_r})$, 
which is zero in degrees $*\ge n+1$.
The analog holds for $S$.
Therefore, in $E_*^{*,*}(S)$ the elements in the 
bottom row $H^{*\ge n+1}(\Z_r)$ must be hit by some differential.
These differentials can come only from the $n$th row at 
the $E_{n+1}$-page (this argument even gives us the $H^*(\Z_r)$-module 
structure of the $n$th row).
Hence there is a non-zero transgressive element $w\in H^0(\Z_r;H^n(S))=H^n(S)^{\Z_r}=E^{0,n}_{2}(S)$, 
that is, $d_{n+1}(w)\neq 0$.
Let $z:=M^*(w)\in H^n(K)^{\Z_r}=E^{0,n}_2(K)$.
Then $d_i(z)=d_i(M^*(w))=M^*(d_i(w))$.
Therefore $z$ survives at least until $E_{n+1}$.
Analogously, the whole $n$th row survives until $E_{n+1}$.
We know that all elements in $E_{n+1}^{*\ge 1,n}(K)$ have 
to die eventually, so they do it exactly on page $E_{n+1}$.
Thus these are exactly the elements that make the part $H^{*\ge n+2}$ 
of the bottom row to zero.

Hence no non-zero differential can arrive at the bottom row on 
an earlier page of $E_*^{*,*}(K)$.
This is because any element $\alpha$ with $d_i(\alpha)=x^ay^b$ 
would imply $d_i(x^k\alpha)=x^{a+k}y^b$ for any $k>0$.
Thus the whole bottom row $H^*(\Z_r)$ of $E_*^{*,*}(K)$ survives until $E_{n+1}$.
At the $E_{n+1}$-page we know exactly the differentials, 
since we know them for the other spectral sequence $E_*^{*,*}(S)$.

Therefore at $E_\infty(K)$, the bottom row is $H^{*\le n}(\Z_r)$. 
Everything else in $E_\infty^{*,0}$ has become zero. Since the index defining map 
$H^*_{\Z_r}(\pt)\to H^*_{\Z_r}(K)$ is the bottom edge homomorphism, we get
\[
\ind_{\Z_r}^\pt(K) = H^{*\ge n+1}(B\Z_r).\vspace{-6mm}
\]
\end{proof}

We apply this theorem to the above maps 
$M:K'\to S(W_r^{d+1})$ and $(M*\id):K'*[r]\to S(W_r^{d+1})*[r]$. 

\begin{corollary}
\label{corIndexOfK}
The $\Z_r$-index of $K'=(\Delta_{r,r-1})^{*(d+1)}$ is
\[
\ind_{\Z_r}^\pt(K')= H^{*\ge N}(B\Z_r)
\]
and the $\Z_r$-index of $K'*[r]$ is
\[
\ind_{\Z_r}^\pt(K'*[r])= H^{*\ge N+1}(B\Z_r).
\]
\end{corollary}

Using this corollary we can compute the index of more 
general joins of chessboard complexes.

\begin{corollary}
\label{corIndexOfDomainOfGeneralTestmap}
Let $0\le c_0,\dots,c_m\le r-1$ and let $s:=\sum c_i$. Let $K:=\Delta_{r,c_0}*\dots*\Delta_{r,c_m}$. Then
\[
\ind_{\Z_r}^\pt(K) = H^{*\ge s}(B\Z_r).
\]
\end{corollary}

\begin{proof} 
Let $L:=\Delta_{r,r-1-c_0}*\dots*\Delta_{r,r-1-c_m}$ and $K':=(\Delta_{r,r-1})^{*(m+1)}$.
Then $\dim K=s-1$  and $\dim K'=(r-1)(m+1)-1$.
We calculate $\dim K'+1=(\dim K +1)+(\dim L +1)$.
There is an inclusion $K'\toG{\Z_r} K*L$.
This implies
\begin{equation}
\label{eqIndABC}
\ind_{\Z_r}(K') \supseteq \ind_{\Z_r}(K*L) \supseteq \ind_{\Z_r}(K)\cdot\ind_{\Z_r}(L) .
\end{equation}
Since $K$ is free, $H^*_{\Z_r}(K)=H^*(K/\Z_r)$, hence
\begin{equation}
\label{eqIndB}
\ind_{\Z_r}(K)\supseteq H^{*\ge \dim K+1}(B\Z_r),
\end{equation}
and analogously
\begin{equation}
\label{eqIndC}
\ind_{\Z_r}(L)\supseteq H^{*\ge \dim L+1}(B\Z_r).
\end{equation}

The dimension $a := \dim K'$ is odd if $r$ is odd. Using Corollary
\ref{corIndexOfK}, we find $\ind_{\Z_r}(K') = H^{\geq a+1}(B\Z_r) = \langle
x^{\frac{a+1}{2}} \rangle$ if $r$ is odd, and $\ind_{\Z_r}(K') =
\langle t^{a+1} \rangle$ if $r=2$.
Together with equation \eqref{eqIndABC}, 
the inclusions \eqref{eqIndB} and \eqref{eqIndC} have to hold with equality.
\end{proof}

It is interesting to note that the last argument of the proof would fail for odd $r$ if $a+1$ were odd, due to the relation $y^2=0$ in $H^*(\Z_r)$.

Now we plug in the configuration space $K$ from \eqref{eqDefOfTestSpaceK} 
and obtain the second proof of Theorem~\ref{thmNewColTopTverThm}.

\medskip

\begin{proof}[Second proof of Theorem \ref{thmNewColTopTverThm}]
By the monotonicity of the index, a test-map $F:K\toG{\Z_r} S(W_r^{d+1})$ 
according to \eqref{eqTestmapF} would imply that
\[
\ind_{\Z_r}^\pt(K)\ \supseteq\ \ind_{\Z_r}^\pt(S(W_r^{d+1})).
\]
This is a contradiction since $\ind_{\Z_r}^\pt(K)= H^{*\ge N+1}(B\Z_r)$ and
$\ind_{\Z_r}^\pt(S(W_r^{d+1}))=H^{*\ge N}(B\Z_r)$, 
as $S(W_r^{d+1})$ is an $(N-1)$-dimensional free $\Z_r$-sphere.
\end{proof}

On the surface this proof seems to be a more difficult reformulation of the first proof.
However, its view point is essential for the transversal generalization, since 
we do not rely on the geometric tools of the
Reduction Lemma \ref{lemReductionToThm21} anymore,
and such a reduction lemma does not seem to exist for the 
Tverberg--\vrecica\ type transversal theorem.
Thus we use the more general configuration space \eqref{eqDefOfTestSpaceK} instead.

\section{A new Borsuk--Ulam type theorem} \label{secNewBUThm}

In this section we prove the second topological main step towards 
the proof of Theorem \ref{thmColoredTV}.
This is the following Borsuk--Ulam type theorem. It will be applied
in combination with the subsequent intersection lemma \ref{lemIntersectionLemma}.

\begin{theorem}[Borsuk--Ulam type]
\label{lemParametrizedBU}
Let
\begin{compactitem}[\ \ $\bullet$]
\item $G=(\Z_p)^m$ be an elementary abelian group $($in particular, $p$ is a prime$)$,
\item $K$ a $G$-CW-complex with index $\ind_G^\pt(K)\subseteq H^{*\ge n+1}(BG;\Z_p)$,
\item $B$ a connected, trivial $G$-space,
\item $E\tof{\phi} B$ a $G$-vector bundle $($all fibers carry the same
 $G$-representation$)$,
\item $\Delta:=E^G\to B$ the fixed-point subbundle of $E\to B$,
\item $C\to B$ its $G$-invariant orthogonal complement subbundle $(E=C\oplus\Delta)$,
\item $F$ be the fiber of the sphere bundle $S(C)\to B$.
\end{compactitem} 
Suppose that 
\begin{compactitem}[\ \ $\bullet$]
\item $n=\rank(C)$,
\item $\pi_1(B)$ acts trivially on $H^*(F;\Z_p)$ $($that is, $C\to B$ is orientable if $p\neq 2)$, and
\item we are given a $G$-bundle map $M$,
\[
\xymatrix{
B\times K\ar[rr]^M \ar[dr]_{pr_1} && E=C\oplus\Delta \ar[dl]^\phi\\
& B &.
}
\]
\end{compactitem} 
Then for $S:=M^{-1}(\Delta)$ and $T:=M(S)=\im(M)\cap \Delta$ 
the maps induced by projection
\[
H^*(B;\Z_p)\ \tof{   pr_1^*}\  H^*_G(S;\Z_p)\textnormal{\qquad and\qquad }
H^*(B;\Z_p)\ \tof{\phi|_T^*}\  H^*(T;\Z_p)
\]
are injective.
\end{theorem}

This theorem generalizes the lemma in Volovikov \cite{Vol96}, which is the special case
when $B=\pt$ and $K$ is $(n-1)$-$\Z_p$-acyclic, and Theorem 4.2 in 
\zivaljevic\ \cite{Ziv99}, whose proof gives the
special case when $m=1$ and $K$ is $(n-1)$-$\Z_p$-acyclic.
The Borsuk--Ulam theorem is the special case $G=\Z_2$, $B=\pt$, $K=S^n$, $E=\R^n$, 
$E$ and $K$ with antipodal action.

We will use it together with the following intersection lemma.

\begin{lemma}[Intersection lemma]
\label{lemIntersectionLemma}
Let $p$ be a positive integer and $\Delta\tof{pr} B$ be a vector bundle 
over an $\Z_p$-orientable compact manifold $B$, whose mod-$p$ Euler class $e\in H^*(B;\Z_p)$ 
satisfies $e^k\neq 0$.
Let $T_0,\dots,T_k\subseteq \Delta$ be sets such that 
$H^{\dim B}(B;\Z_p)\tof{pr_*} H^{\dim B}(T_i;\Z_p)$ is injective for all ~$i$.
Then $T_0\cap\dots\cap T_k\neq \emptyset$.
\end{lemma}

A proof for the case where $p$ is prime can be extracted from \cite{Ziv99}. 
For non-primes $p$ one needs to change \zivaljevic's argument only at the point 
where he needs $H_{\dim B}(V_i;\Z_p)\to H_{\dim B}(B;\Z_p)$ to be surjective for 
a neighborhood $V_i$ of $T_i$. This can be proven for $\Z_p$-coefficients 
by a (non-standard) universal coefficient theorem that 
computes homology from cohomology, using $\Z_p$ as the base ring and the fact that $\Z_p$ is an injective $\Z_p$-module.

Next we prove our parameterized Borsuk--Ulam type theorem. We will use \cech\ cohomology 
with $\Z_p$-coefficients.

\medskip
 
\begin{proof}[Proof of Theorem \ref{lemParametrizedBU}] 
~\\
\textbf{(1.)} Let $C\to B$ be the orthogonal complement bundle of $\Delta$ 
in $E$ with respect to a $G$-invariant metric (such that $C$ becomes a $G$-bundle).
We denoted by $S(C)$ the associated sphere bundle and by $F$ the fiber 
in $S(C)$ over some base point $b\in B$.
We denote by $E_*^{*,*}(F)$ and $E_*^{*,*}(S(C))$ the Leray--Serre spectral sequences 
associated to the fibrations 
\begin{equation}
\label{eqBundleF}
F\incl EG\times_G F\to BG\times b
\end{equation}
and 
\begin{equation}
\label{eqBundleSofC}
F\incl EG\times_G S(C)\to BG\times B,
\end{equation}
respectively, see Figure \ref{figSSforBUthm}.

The $E_2$-page $E_2^{*,*}(F)$ has only two non-zero rows, the $0$-row and the $(n-1)$-row.
The local coefficients in $E^{p,q}_2(F)=H^p(BG,H^q(F))$ 
are given by the $\pi_1(BG)$-module structure on $H^q(F)$.
Since $G=\pi_1(BG)$ is an elementary abelian group and $F$ is a sphere, 
the $H^*(BG)$-module structure on $H^q(F)$ is trivial, for the 
$G$ action on $H^q(F)$ is induced by homeomorphisms $F\to F$, and the 
degree of this homeomorphism has to be $1$ if $p$ is odd.
Therefore
\[
E^{p,q}_2(F)=H^p(BG,H^q(F))=H^p(BG)\otimes H^q(F).
\]
The differentials are $H^*(BG)$-homomorphism, and
\begin{equation*}
E^{*,n-1}_n(F)=E^{*,n-1}_{2}(F)=H^*(BG)\otimes H^{n-1}(F)
\end{equation*}
is a $H^*(BG)$-module generated by
\begin{equation*}
1\in E^{0,n-1}_{n}(F)=H^0(BG)\otimes H^{n-1}(F),
\end{equation*}
where $1$ is regarded as the generator of $H^{0}(BG)$.
Hence there is a non-vanishing differential in $E_*^{*,*}(F)$ 
if and only if the differential $d_n: E^{0,n-1}(F)\to E^{n,0}(F)$ is non-zero.
Since $F$ is fixed-point free, the edge homomorphism $H^*(BG)\to H^*_G(F)$ 
is not injective \cite[Prop.~3.14, p.~196]{Die87},  
thus there must be a non-vanishing differential.
Therefore there is a non-zero element $\alpha=d_{n}(1)\in\ind_G^\pt(F)$ of degree $n$.
\smallskip

\noindent
\textbf{(2.)} Now the inclusion $F\incl S(C)$ gives a bundle map from \eqref{eqBundleF} to \eqref{eqBundleSofC},
\begin{equation}
\label{eqBundleMapInclusionOfFibre}
\xymatrix{ EG\times_G F\ar[r]\ar[d]&EG\times_G S(C)\ar[d]
\\
BG\times b\ar[r]&BG\times B}
\end{equation}
which induces a morphism of associated Lerray--Serre spectral sequences 
$E_*^{*,*}(S(C))\to E_*^{*,*}(F)$, see Figure~\ref{figSSforBUthm}.

\begin{figure}[tbh]
\centering
\input{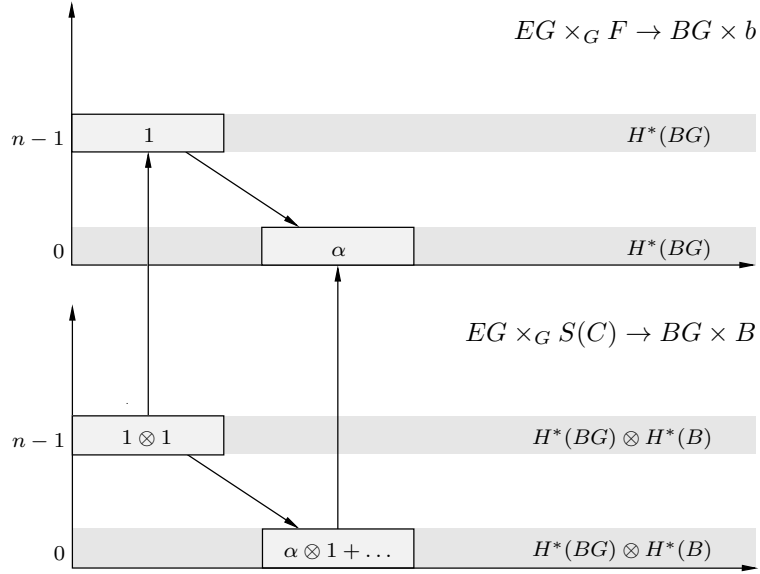} 
\caption{The morphism of spectral sequences induced 
by the bundle map \eqref{eqBundleMapInclusionOfFibre}.}
\label{figSSforBUthm}
\end{figure}

\smallskip
 
The $E_2$-page of $E_*^{*,*}(S(C))$ is $E^{p,q}_2(S(C))=H^p(BG\times B,H^q(F))$ 
where the local coefficients are given by the $\pi_1(BG\times B)$-module structure 
on $H^q(F)$.
Since $H^*(F)$ is a trivial $G\times \pi_1(B)$-module, the $0$- and $(n-1)$-rows 
of this spectral sequence are given by
\[
E^{p,q}_2(S) = H^p(BG\times B;H^q(F)) = H^p(BG\times B)=\bigoplus_{i=0}^{p}H^i(BG)\otimes H^{p-i}(B),\ \textnormal{for }q\in\{0,n-1\}.
\]
The morphism of the spectral sequences $E_*^{*,*}(S(C))\to E_*^{*,*}(F)$ 
on the $0$-row and on the $(n-1)$-row of the $E_2$-page,
\[
\bigoplus_{i=0}^{p}H^i(BG)\otimes H^{p-i}(B)\rightarrow H^p(BG),
\]
is zero on $\bigoplus_{i=1}^{p}H^i(BG)\otimes H^{p-i}(B)$. 
On $H^p(BG)\otimes H^0(B)=H^p(BG)$ it is just the identity.
The differential of the generator 
$1\otimes 1\in H^0(BG)\otimes H^0(B)$  
of $E_n^{0,n-1}(S(C))$
hits an element $\gamma\in E_n^{n-1,0}(S(C))$ of 
the bottom row $\bigoplus_{i=0}^{n}H^i(BG)\otimes H^{n-i}(B)$.
Since the differentials commute with morphisms of spectral sequences, 
$\gamma$ is an element in $\ind_G^B(S(C))\subseteq H^*(BG)\otimes H^*(B)$ 
that restricts to $\alpha$ under the map 
$\bigoplus_{i=0}^{n}H^i(BG)\otimes H^{n-i}(B)\rightarrow H^{n}(BG)$, hence $\gamma\neq 0$.
Since $\alpha$ and $\gamma$ are of degree $n$, $\gamma$ 
has the form $\alpha\otimes 1+\sum_i \delta_i\otimes\eps_i$, where 
$\deg\delta_i+\deg\eps_i=n$ and $\deg\delta_i\le n-1$.
\smallskip

\noindent
\textbf{(3.)} Formula \eqref{eqIndexThm} of Section \ref{secNewColTopTverThmRevisited} yields
\[
\ind_G^B(S)\cdot\ind_G^B(S(C))\subseteq\ind_G^B(B\times K)=\ind_G^\pt(K) \otimes H^*(B).
\]
We know that $\ind_G^\pt(K)\subseteq H^{*\ge n+1}(BG)$, and in 2.) 
we got an element $\gamma=\alpha\otimes 1+\sum_i \delta_i\otimes\eps_i\in \ind_G^B(S(C))$.
Therefore, $\ind_G^B(S)\subseteq H^*(BG)\otimes H^*(B)$ 
does not contain any non-zero element of the form $1\otimes\beta$, 
$\beta\in H^*(B)$.
Indeed, if $1\otimes\beta\in \ind_G^B(S)\wo\{0\}$, then
\[
(1\otimes\beta)\cdot\gamma=\alpha\otimes\beta+\sum_i
\delta_i\otimes(\beta\cdot\eps_i).
\]
Since $\deg(\delta_i)<\deg(\alpha)=n$, this implies that $\alpha\in\ind_G^\pt(K)$.
This contradicts $\ind_G^\pt(K)\subseteq H^{*\ge n+1}(BG)$.

Hence the following composition is injective,
\[
H^*(B)\tof{1\otimes \id} H^*(BG)\otimes H^*(B)\to H^*_G(S),
\]
where both maps are induced by projection.
The following diagram is induced by the obvious maps
\[
\xymatrix{
H^*_G(S)
&H^*_G(T)=H^*(BG)\otimes H^*(T)  \ar[l]
&H^*(T) \ar[l]
\\
&\ar[ul]^{\textnormal{inj.}}\ar[u]\ar[ur] H^*(B)}    
\]
which shows that $H^*(B)\to H^*(T)$ has to be injective as well.
\end{proof}

\section{Proof of the colored Tverberg--\vrecica\ type theorem}\label{secNewColTVThm} 

Now we use all the topological tools that we have developped to prove our Main Theorem.

\medskip

\begin{proof}[Proof of Theorem \ref{thmColoredTV}]
First we will assume that $p=2$ or that $d$ and $k$ are odd.
The remaining case, when $p$ is odd and $d$ and $k$ are even, will be shown by 
an elementary reduction   
at the end of the proof.

There is a natural proof scheme that was already used in \cite{Ziv99}, 
\cite{Vre08} and \cite{Kar07}, namely the configuration space test map method.
Our progress in this paper stems from the topological index calculations of 
Section \ref{secNewColTopTverThmRevisited} and the Borsuk-Ulam type 
Theorem \ref{lemParametrizedBU} in Section \ref{secNewBUThm}.

Suppose we are given $\CC^\ell$s and $r$ as in Theorem \ref{thmColoredTV}, 
together with the colorings
\[
\CC^\ell\ =\ \biguplus_{i=0}^{m_\ell} C_i^\ell.
\]
A partition of the sets $\CC^\ell$ into $r$ colorful pieces $F_j^\ell$ admits a $k$-plane $P$ 
intersecting all convex hulls of the pieces $F_j^\ell$ if and only if one can project 
the pieces orthogonally to a $(d-k)$-dimensional subspace of $\R^d$ (namely the 
orthogonal complement of $P$) such that the convex hulls of the projected $F_j^\ell$s
have a point in common.
Calculations turn out to be easier if we look first at the set 
of colored Tverberg points of all projections of one single fixed $\CC^\ell$ 
and then show that the corresponding sets for all $\CC^\ell$s have to intersect.

Fix an $\ell\in\{0,1,\dots,k\}$.
Let $B:=G_{d,d-k}$ be the Grassmannian manifold of all $(d-k)$-dimensional subspaces 
of $\R^d$ and   $\gamma\to B$   the tautological bundle over $B$.
Let $G:=\Z_r$ act on $[r]$ by left translations.
Let $E:=\gamma^{\oplus r}$.
$E$ is a $G$-bundle over $B$ whose fixed-point subbundle $\Delta:=E^G\iso\gamma$ 
is the thin diagonal bundle.

The space
\begin{equation}
\label{eqDefOfCSforCTV} 
K\ :=\ \Delta_{r,|C_0^\ell|} * \dots * 
       \Delta_{r,|C_{m_\ell}^\ell|}\subseteq (\sigma_{|\CC^\ell|-1})^{*r}
\end{equation}
will again be the configuration space.
For each $b\in B$, we can project $\CC^\ell$ to the $(d-k)$-space given by~$b$, 
which can be identified with the fiber over~$b$ in~$\gamma$.
This projection can be extended affine linearly to a map from $\Delta^{|\CC^\ell|-1}$ 
to the fiber.
Doing this componentwise and for all $b\in B$, we get a bundle map
\[
B\times K\tof{M} E.
\]
Define $T^\ell:=\im(M)\cap\Delta$, which is the set of colored Tverberg points 
of the respective projected sets $\CC^\ell$. Each point of $T^\ell\subseteq \Delta$, 
which lies in the fiber over say $b\in B$, lies in the intersection of the 
convex hulls of some colored $r$-partition of the set $T^\ell$ that got projected to~$b$.
Hence we need to show that $T^0\cap\dots\cap T^k\neq\emptyset$.
We will apply the calculations and tools of the previous sections.

$K$ is of dimension $(r-1)(d-k)$.
The ranks of $E$ and $\Delta$ are $r(d-k)$ and $d-k$. We claim that the 
orthogonal complement bundle $C$ of $\Delta$ in $E$ is $\Z_r$-orientable.
Since all vector bundles are $\Z_2$-orientable, we only need to deal 
with the case where $r$ is odd.
Then $r-1$ is even and $C$ is stably isomorphic to $\gamma^{r-1}$, 
which is an even power of a bundle, hence orientable.
Therefore we can apply the Borsuk--Ulam type Theorem \ref{lemParametrizedBU} 
and get that $H^*(B)\to H^*(T_i)$ is injective.
To apply the Intersection Lemma \ref{lemIntersectionLemma}, 
we need that $e^k\neq 0$ for the mod-$r$ Euler class $e\in H^{d-k}(B)$ of $\Delta\iso\gamma$.

If $r=2$ then $e$ is the top Stiefel--Whitney class $w_{d-k}$, 
whose $k$-th power is the mod-$2$ fundamental class of $B$ (see e.g.\ \cite[Lemma 1.2]{Hil80}), 
and we proved the theorem.
Now we come to the case where $r$ is odd.
If $\rank(\gamma)=d-k$ is odd then the mod-$r$ Euler class is zero and we 
cannot deal with this case.
If $d-k$ is even then we may assume that $d$ and $k$ are odd, otherwise we prove the 
theorem for $d'=d+1$ and $k'=k+1$ first and use the 
reduction lemma \ref{lemElementaryReductionLemma} below afterwards.
Then $e^k\neq 0$ is readily proved in Proposition 4.9 in \cite{Ziv99}, based on \cite{FH88}.
In fact, he even shows it for the tautological bundle over the \emph{oriented} Grassmannian.
Since this bundle is the pullback of $\gamma$ we are done by naturality of the Euler class.

Finally we prove the elementary Reduction Lemma \ref{lemElementaryReductionLemma} 
that also gives the case when $p$ is odd and $d$ and $k$ are even.

\begin{lemma}[Reduction Lemma]
\label{lemElementaryReductionLemma}
If Conjecture \ref{conjColoredTV} holds for parameters $(d,k,r_0,\dots,r_k)$ 
then so it does for
$(d',k',r_0,\dots,r_{k-1})$ with $d':=d-1$ and $k':=\max(k-1,0)$.
\end{lemma}

\begin{proof}
We will prove only the case $k\ge 1$, since the case $k=0$ is exactly the 
reduction that is used in the proof of Lemma~\ref{lemReductionToThm21} \cite{BMZ09}.

Assume that Conjecture \ref{conjColoredTV} is true for parameters $(d,k,r_0,\dots,r_k)$ 
and suppose we are given colored sets $\CC^0,\dots,\CC^{k-1}\subseteq\R^{d-1}$ 
where we have 
to partition $\CC^\ell$ into $r_\ell$ pieces such that some $(k-1)$-dimensional plane meets the 
convex hulls of all pieces.
To do this, view $\R^{d-1}$ as the hyperplane in $\R^d$ where the last coordinate 
is zero, and define $\CC^k\subset\R^d$ to be a set of $(r_k-1)(d-k+1)+1$ points 
all of which lie close to $(0,\dots,0,1)$.
We color $\CC^k$ in an arbitrary way. For example, we may give each point a different color.
Since Conjecture \ref{conjColoredTV} is true for $(d,k,r_0,\dots,r_k)$, 
we can partition the sets $\CC^\ell$ appropriately and find a $k$-plane~$P$ 
meeting all of the convex hulls of the pieces.
Since $P$ goes through the convex hull of $\CC^k$, it cannot be fully contained in $\R^{d-1}$.
Therefore $P\cap\R^{d-1}$ is a $(k-1)$-plane intersecting the convex hulls 
of the pieces of the sets $\CC^0,\dots,\CC^{k-1}$.
\end{proof}

This finishes the proof of the Main Theorem.
\end{proof}

\medskip

The proof immediately extends to a topological version of it, the Main Theorem being the special case where all maps $f_\ell$ are affine.

\begin{theorem} [Topological Main Theorem]
\label{thmTopColoredTV}
Let $r$ be prime and $0\le k\le d$ such that $r(d-k)$ is even or $k=0$.
Let $\CC^\ell$\ $(\ell=0,\dots,k)$ be sets of cardinality $|\CC^\ell|=(r-1)(d-k+1)+1$, 
which we identify with the vertex sets of simplices $\simplex_{|\CC^\ell|-1}$.
We color them
\[
\CC^\ell\ =\ \biguplus C^\ell_i,
\]
such that no color class is too large, $|C^\ell_i|\le r-1$.
Let
\[
f_\ell: \simplex_{|\CC^\ell|-1}\to \R^d
\]
be continuous maps.
Then we can partition each $\CC^\ell$ into colorful sets $F^\ell_1,\dots,F^\ell_{r}$ 
and find a $k$-plane $P$ that intersects all the sets $f_\ell(\conv(F_j^\ell))$.
\end{theorem}

Unfortunately our proof of the Main Theorem does not extend to prime powers 
$r_\ell=p^{a_\ell}$ over the same prime $p$.
The basic reason is that the degree of $M$ vanishes modulo $r$ if and only if $r$ 
divides $(r-1)!^d$ (see~\eqref{eqDegreeOfM}). 
Therefore this proof can only work if $r$ is a prime or if $r=4$ and $d=1$.
For $k=0$, even using the full symmetry group $\Sym_r$ does not help since 
an $\Sym_r$-equivariant test-map exists if and only if $r$ divides $(r-1)!^d$;
see~\cite{BMZ09}. To see this one needs to take a closer look at the obstruction, 
the degree proof from Section \ref{secNewColTopTverThmRevisited} does not extend immediately.

\subsubsection*{A new proof for Karasev's result \cite{Kar07}}
We can extend the above proof to arbitrary powers of a fixed prime $p$ 
if we do not require the colored Tverberg partitions to be colored, 
or in other words, if all color classes are singletons.
In this case, the configuration space $K$ of \eqref{eqDefOfCSforCTV} becomes
\[
K=[r_\ell]^{*(N+1)} = (\sigma_N)^{*r_\ell}_{\Delta(2)},
\]
which is the $r_\ell$-wise $2$-fold deleted join of an $N$-simplex. 
\barany, Shlosman and Sz\"ucs \cite{BSS81} proved that it is $(N-1)$-connected. 
This also follows from the connectivity relation $\conn(A*B)\ge \conn(A)+\conn(B)+2$ 
for $CW$-complexes.
As symmetry group we take instead of $\Z_r$ a subgroup $G\iso (\Z_p)^{m_\ell}$ 
of $S_{r_\ell}$ that acts fixed-point freely on $[r_\ell]$.
By the connectivity of $K$, $\ind_G^\pt(K)\subseteq H^{*\ge N+1}(BG)$, 
as we can directly deduce from the Leray--Serre spectral sequence 
of $K\incl EG\times_G K\to BG$.
We obtain Karasev's result.

\begin{theorem}[\cite{Kar07}]
The Tverberg--\vrecica\ Conjecture \ref{conjTV} holds 
in the special case $r_\ell=p^{a_\ell}$, where $p$ is a prime such that 
$p(d-k)$ is even or $k=0$.
\end{theorem}

\subsubsection*{Tightness of the Main Theorem \ref{thmColoredTV}}
 
\begin{observation}
For any $0\le k\le d-1$, $r_\ell\ge 2$ and $0\le \ell^*\le k$, we can choose 
point sets $\CC^\ell\subset\R^d$, $0\le \ell\le k$, of size $|\CC^\ell|=(r_\ell-1)(d-k+1)+1$ 
and color them with one color class $C^{\ell^*}_j$ of size $r_{\ell^*}$ and all 
other color classes as singletons such that there are no colorful partitions 
of the $\CC^\ell$s into $r_\ell$ pieces each that admit a common $k$-dimensional transversal.
\end{observation}

\begin{proof}
Let $V^\ell$, $0\le \ell\le k$, be pairwise parallel $(d-k)$-dimensional affine subspaces of $\R^d$ 
such that their projections to an orthogonal $k$-space are the $k+1$ vertices of a $k$-simplex.
On each $V^\ell$ we place a standard point configuration $\CC^\ell$:
Take a $(d-k)$-simplex $\sigma^\ell$ in $V^\ell$, let $\CC^\ell$ have 
$r_\ell-1$ points on each vertex 
of $\sigma^\ell$ and put the last vertex of $\CC^\ell$ into the center $c^\ell$ of $\sigma^\ell$.

The sets $\CC^\ell$ admit only one Tverberg point, 
namely $c^\ell$.
Hence a potential common $k$-dimensional transversal $P$ must intersect all $c^\ell$. Since the 
$V^\ell$ have been chosen generically enough, 
$P$ is uniquely determined and $P\cap V^\ell=\{c^\ell\}$.

Now we color the points of an arbitrary $\CC^\ell$ at an arbitrary vertex of $\sigma^\ell$ red, 
together with a further point at another vertex.
Even if all other color classes in $\CC^\ell$ are singletons there will be no colored 
Tverberg partition of $\CC^\ell$. Together with $P\cap V^\ell=\{c^\ell\}$ this proves the observation.
\end{proof}

\medskip

\noindent
\textbf{Acknowledgements.} We would like to thank Sini\v{s}a \vrecica\ and Rade \zivaljevic,
who independently discovered the degree proof for \cite[Theorem 2.2]{BMZ09},
for their interest and their enthusiasm. 
We are grateful to Aleksandra, Julia and Torsten for constant support.

\end{document}